\documentclass[10pt]{amsart}
\usepackage{amssymb,amsmath,amsopn,textcomp, thmtools}
\usepackage[foot]{amsaddr}
\usepackage{aliascnt}
\usepackage{mathrsfs} 

\theoremstyle{plain}
\newtheorem{theorem}{Theorem}[section]

\newtheorem{lemma}[theorem]{Lemma}

\newtheorem{corollary}[theorem]{Corollary}
\newtheorem{proposition}[theorem]{Proposition}
\theoremstyle{definition}
\newtheorem{remark}[theorem]{Remark}

\newcommand{\iR}{\mathbb{R}}

\newcommand{\iZ}{\mathbb{Z}}

\usepackage{bbm}
\newcommand{\R}{\mathbb{R}}

\usepackage{extarrows}
\usepackage{dsfont}
\usepackage{tikz,subcaption} 

\usepackage{mathabx} 

\allowdisplaybreaks

\begin{document}

\title{Li-Yau inequalities for general non-local diffusion equations via reduction to the heat kernel}
\date{\today}
\author{Frederic Weber}
\email{frederic.weber@uni-ulm.de}
\author{Rico Zacher$^*$}
\thanks{$^*$Corresponding author. F.\ W.\ is supported by a PhD-scholarship of the ``Studienstiftung des
deutschen Volkes'', Germany. R.\ Z.\ is supported by the DFG (project number 355354916, GZ ZA 547/4-2).}
\email[Corresponding author:]{rico.zacher@uni-ulm.de}
\address[Frederic Weber, Rico Zacher]{Institut f\"ur Angewandte Analysis, Universit\"at Ulm, Helmholtzstra\ss{}e 18, 89081 Ulm, Germany.}


\begin{abstract} We establish a reduction principle to derive Li-Yau inequalities for non-local diffusion problems in a very
general framework, which covers both the discrete and continuous setting. Our approach is not based on curvature-dimension inequalities but on heat kernel representations of the solutions and consists in reducing the problem to the heat kernel.
As an important application we solve a long-standing open problem by obtaining a Li-Yau inequality for positive solutions $u$ to the fractional (in space) heat equation of the form
$(-\Delta)^{\beta/2}(\log u)\le C/t$, where $\beta\in (0,2)$. We also show that this Li-Yau inequality allows to derive
a Harnack inequality. We further illustrate our general result with an example in the discrete setting by proving a sharp Li-Yau inequality for diffusion on a complete graph. 
\end{abstract}

\maketitle

\bigskip
\noindent \textbf{Keywords:} Li-Yau inequality, differential Harnack inequality, Harnack inequality, non-local diffusion, heat kernel, fractional Laplacian, continuous-time Markov chains, complete graphs
 
 \noindent \textbf{MSC(2020)}: 35R11, 35K08, 60J27, 39A12

\section{Introduction}
The celebrated Li-Yau inequality states that on a complete Riemannian manifold $M$ with topological dimension $n$ and non-negative Ricci curvature, we have
\begin{equation}\label{eq:classicalLiYau}
- \Delta (\log u) \leq \frac{n}{2t},\quad t>0,
\end{equation}
for any positive solution $u$ to the heat equation $\partial_t u = \Delta u$ on $(0,\infty)\times M$, where $\Delta$ denotes the Laplace-Beltrami operator. This inequality originates from the seminal work \cite{LY}. It is optimal in the sense that equality is achieved for the heat kernel in the Euclidean case. As an important application of \eqref{eq:classicalLiYau} (sharp) parabolic Harnack inequalities can be deduced. Using the $\Gamma$-calculus of Bakry and \'Emery, the Li-Yau inequality \eqref{eq:classicalLiYau} has been generalized to Markov diffusion operators that satisfy the curvature-dimension condition $CD(0,n)$, see \cite{BBG} and also the extensive monograph \cite{BGL}. 

Concerning the non-local situation, the approach via curvature-dimension inequalities has stimulated a lot of research in the {\em discrete} setting. Based on discrete replacements of certain chain rule identities, in \cite{BHLLMY15,DKZ,Mun} suitable substitutes of curvature-dimension conditions have been introduced in order to prove Li-Yau type inequalities for generalized graph Laplacians. Corresponding Li-Yau inequalities for non-local operators describing diffusion processes with arbitrary long jumps do not seem to be known.

In the continuous setting, the fractional Laplacian on $\R^d$ is one of the most prominent representatives of a non-local operator. It is a natural question to ask in which way properties for the Laplace operator in the Euclidean case carry over to the fractional Laplacian. In particular it is of great interest whether a Li-Yau type inequality is also valid for positive solutions of the fractional heat equation, and if yes, whether the time-dependence is like in \eqref{eq:classicalLiYau}. The recent survey article \cite{Gar} of Garofalo dedicates a whole chapter to address these open questions.

The main aim of this article is to establish a general reduction principle to derive Li-Yau inequalities,  that applies to the discrete and continuous setting, respectively, and which answers both of the above questions for the fractional heat equation positively. For that purpose, we will not follow the approach of curvature-dimension inequalities. Remarkably, in the quite recent article \cite{SWZ} it has been shown that  the fractional Laplacian does not satisfy $CD(\kappa,n)$ for any pair $(\kappa,n)\in \R \times (0,\infty)$, which is in sharp contrast to the Euclidean Laplacian. Instead, we follow the different approach of reducing the problem to the heat kernel. In the case of the classical heat equation on $\R^d$, the heat kernel satisfies \eqref{eq:classicalLiYau} with equality and one is able to deduce the Li-Yau inequality \eqref{eq:classicalLiYau} by using that positive solutions to the heat equation are given as the convolution of the heat kernel with the respective initial datum, see e.g.\ \cite[Section 21]{Gar}. We show in a very general framework that this reduction principle also works for positive solutions to the corresponding non-local diffusion equation. Indeed, by Theorem \ref{thm:abstractmainresult} we conclude that the heat kernel determines the corresponding Li-Yau inequality for solutions that can be represented via a respective integral representation involving the corresponding heat kernel. As an application, we apply this principle to the fractional heat equation and derive a Li-Yau inequality (see Theorem \ref{thm:LiYaufrak}), which is optimal with respect to the time-behaviour. To highlight that our method applies to a quite general setting, we also illustrate the case of a discrete state space in Section \ref{sec:discrete}. Here, we consider the example that the underlying graph is given by the complete graph $K_n$, $n\geq 2$, and deduce an optimal  Li-Yau type inequality, which improves known results from \cite{DKZ}. 

As we have mentioned before, parabolic Harnack inequalities belong to the most important applications of the original Li-Yau estimate. Therefore it is natural to ask whether non-local Li-Yau inequalities lead to similar applications. While in the discrete setting this has been positively answered by \cite{BHLLMY15, DKZ, Mun}, we are not aware of any proof of a Harnack inequality via a Li-Yau type inequality for an evolution equation with a purely non-local diffusion operator in the space-continuous setting. In Theorem \ref{HarnackFracEvol} we derive a scale-invariant Harnack inequality for the fractional heat equation, where the crucial point
is not the result as such (in fact, better results already exist in the literature, see Remark \ref{ClassicalHarnack} below), but the fact that such a derivation is possible at all. In contrast to the classical heat 
equation, we do not have a gradient term $|\nabla( \log u)|^2$ in the corresponding differential Harnack inequality, but have
 to work with the non-local operator $\Psi_\Upsilon(\log u)$ (see \eqref{eq:PsiPalme})
involving the function $\Upsilon(z)=e^z-1-z$, which makes the proof much more difficult. 
Here we point out that one cannot argue as in the discrete case, where this operator already appeared, see the work \cite{DKZ} on Li-Yau inequalities
in the discrete setting. We remark that the operator $\Psi_\Upsilon$ also plays a fundamental role in the context of curvature-dimension conditions for non-local
operators including the discrete case of Markov chains, see \cite{WZ}.

The article is organized as follows. In the next section we establish a reduction principle to derive Li-Yau inequalities for non-local diffusion problems in a very general framework. Thereafter, we apply this principle in Section \ref{sec:frakLaplace} to the fractional heat equation and in Section \ref{sec:discrete} to the diffusion equation with the Markov generator that corresponds to the complete graph. Section \ref{HarnackInequ} is devoted to the derivation of Harnack inequalities by means of the Li-Yau estimates obtained before.

\medskip
We have been informed that Tuhin Ghosh and Moritz Kassmann from Bielefeld (Germany) have recently
proved a different version of Li-Yau type inequality for the fractional heat equation by an alternative approach, see
\cite{GhKa}. Their result also allows to derive a Harnack inequality. A preprint shall be available soon.
\section{A reduction principle for deriving Li-Yau inequalities}\label{sec:abstractsection}
Let $(M,d)$ be a metric space and $\mathcal{B}(M)$ denote the Borel $\sigma$-algebra on $M$. We consider a non-local operator of the form
\begin{equation}\label{eq:MMSnonlocaloperator}
L f(x) = \int_{M\setminus \{x\}} \big( f(y)-f(x)\big) k(x,\mathrm{d}y),
\end{equation}
where the kernel is such that $k(x,\cdot)$ defines a $\sigma$-finite measure  on $\mathcal{B}(M\setminus\{x\})$ for any $x\in M$ and $f:M\to \R$ is such that the integral exists. We also include the case where the integral on the right-hand side of \eqref{eq:MMSnonlocaloperator} is singular. In this situation one replaces $\int_{M \setminus \{x\}}$ by $\lim\limits_{\varepsilon \to 0^+} \int_{M \setminus B_\varepsilon(x)}$ in the right-hand side of \eqref{eq:MMSnonlocaloperator}. This is motivated by the important example of the fractional Laplace operator, where we choose $M=\R^n$ and $d$ as the corresponding Euclidean distance.  However, the quite general formulation also covers a variety of other important situations, such as, for instance,  
the case where $L$ is the generator of a continuous-time Markov chain on a discrete state space
endowed with the natural graph structure with $M$ being the set of vertices and edge weights given by the corresponding transition rates between two respective states.

In \cite{DKZ}, the following formula serves as a replacement of a classical chain rule. It has been established in the context of locally finite graphs and also for a class of non-local operators on $\R^n$ that are of the form \eqref{eq:MMSnonlocaloperator}. We state the result here again due to our more general setting. The short proof follows analogously as in \cite{DKZ}.
\begin{lemma}\label{lem:Hfundametalidentity}
Let $D\subset \R$ be an open set, $h \in C^1(D;\R)$, $f:M \to D$ and $x\in M$ such that $Lf(x)$ and $L(h(f))(x)$ exist. Then we have
\begin{equation*}
L \big( h(f)\big)(x) = h'(f(x)) L f(x) + \int_{M \setminus \{x\}} \Lambda_h\big(f(y),f(x)\big)k(x,\mathrm{d}y),
\end{equation*}
where $\Lambda_h(w,z):= h(w)-h(z)-h'(z)(w-z)$ for any $w,z \in D$.
\end{lemma}
In the context of this article, we will apply Lemma \ref{lem:Hfundametalidentity} with the specific choice of $h(r)=\log(r)$, $r \in (0,\infty)$. In this case we have
\begin{equation*}
\Lambda_{\log}(w,z)= \log w - \log z -\frac{w-z}{z} = -\Upsilon(\log w - \log z), \, w,z \in (0,\infty),
\end{equation*}
where $\Upsilon(r):= e^r - r -1$, $r\in \R$. Introducing the 
operator 
\begin{equation}\label{eq:PsiPalme}
\Psi_\Upsilon(f)(x) = \int_{M \setminus \{x\}} \Upsilon \big( f(y)-f(x)\big) k(x,\mathrm{d}y),
\end{equation}
we deduce from Lemma \ref{lem:Hfundametalidentity} that
\begin{equation}\label{eq:fundamentalidentity}
L (\log f) =\frac{L f }{f} - \Psi_\Upsilon(\log f)
\end{equation}
holds for any positive function $f$ whenever $L f$ and $L (\log f)$ exist. Note that a possible singularity at $y=x$ in the right-hand side of \eqref{eq:PsiPalme} does not play a role by positivity of $\Upsilon$ in the sense that the integral is either finite or $\infty$. In the particular case of the fractional Laplacian, a simple Taylor argument shows that the quadratic behavior of 
$\Upsilon$ near the origin ensures that the singularity of the corresponding integral kernel is compensated provided the function is sufficiently smooth. 

Now, we establish the key estimate of this article.
\begin{lemma}\label{lem:generalintegralestimate}
Let $L$ be an operator of the form \eqref{eq:MMSnonlocaloperator} and let $H:M\times M \to (0,\infty)$ be such that $H(x,\cdot)$ is $\mathcal{B}(M)$-measurable and the restriction  $\left.H\right|_{M\setminus \{x\} \times M}$  is $\mathcal{B}(M\setminus\{x\}) \otimes \mathcal{B}(M)$-measurable for any $x \in M$, respectively. Further, let $f:M\to (0,\infty)$  be $\mathcal{B}(M)$-measurable. We assume that the integral $P f(x)= \int_{M} H(x,y)f(y)\,\mathrm{d}\nu(y)$ and also $\Psi_\Upsilon(\log P f)$ exist for any $x \in M$ and that for $\nu$-a.e. $y \in M$ the expression $\Psi_\Upsilon(\log H(\cdot,y))(x)$ exists for every  $x \in M$.  Here $\nu:\mathcal{B}(M)\to [0,\infty]$ is a $\sigma$-finite measure. Then we have
\begin{equation} \label{KEY}
\int_{M} \Psi_\Upsilon (\log H(\cdot,y))(x)H(x,y)f(y) \mathrm{d}\nu(y) \geq \Psi_\Upsilon(\log P f)(x) P f(x).
\end{equation}
\end{lemma}
\begin{proof}
Recalling \eqref{eq:PsiPalme}, we write
\begin{align*}
\Psi_\Upsilon(\log P f)(x) P f (x) &= \int_{M\setminus\{x\}}  \Upsilon\Big(\log \frac{P f(h)}{P f(x)}\Big)P f(x)k(x,\mathrm{d}h) \\
&= \int_{M\setminus\{x\}}\Upsilon\Big(\log \frac{P f(h)}{P f(x)}\Big) \int_{M} H(x,y)f(y)\mathrm{d}\nu(y)\,k(x,\mathrm{d}h).
\end{align*}
Next, by Tonelli's theorem we observe that
\begin{align*}
\int_{M} \Psi_\Upsilon\big(\log H(\cdot,y)) (x) & H(x,y) f(y)\mathrm{d}\nu(y) \\
&= \int_{M} H(x,y) f(y) \int_{M\setminus\{x\}}  \Upsilon\Big(\log \frac{H(h,y)}{H(x,y)}\Big)k(x,\mathrm{d}h) \, \mathrm{d}\nu(y) \\
&= \int_{M\setminus\{x\}}\int_{M} H(x,y) f(y)\Upsilon\Big( \log \frac{H(h,y)}{H(x,y)}\Big)\mathrm{d}\nu(y)\,k(x,\mathrm{d}h).
\end{align*}
Consequently, we have
\begin{equation}\label{eq:Hlhsbeforeconvexity}
\begin{split}
\int_{M} \Psi_\Upsilon\big(\log H(\cdot,y)) &(x) H(x,y) f(y)\mathrm{d}\nu(y) - \Psi_\Upsilon(\log P f)(x) P f (x) \\
&= \int_{M\setminus\{x\}} \int_{M} \Big(\Upsilon\Big( \log \frac{H(h,y)}{H(x,y)}\Big) - \Upsilon\Big(\log \frac{P f(h)}{P f(x)}\Big)\Big) H(x,y)f(y)\mathrm{d}\nu(y)\, k(x,\mathrm{d}h). 
\end{split}
\end{equation}
Now, it can be readily checked that the mapping $r\mapsto \Upsilon(\log(r))$, $r \in (0,\infty)$, is convex. Since 
$\frac{\mathrm{d}}{\mathrm{d}r}\Upsilon(\log(r))=\frac{r-1}{r}$, $r\in (0,\infty)$, we infer from convexity that
\begin{equation*}
\Upsilon(\log(r))-\Upsilon(\log(s)) \geq \frac{s-1}{s}(r-s)
\end{equation*}
holds for any $r,s \in (0,\infty)$. With this at hand and using the positivity of $f$ and $H$, we can now estimate the right-hand side of \eqref{eq:Hlhsbeforeconvexity} as follows.
\begin{equation*}
\begin{split}
&\int_{M \setminus\{x\}} \int_{M}\Big(\Upsilon\Big( \log \frac{H(h,y)}{H(x,y)}\Big) - \Upsilon\Big(\log \frac{P f(h)}{P f(x)}\Big)\Big) H(x,y)f(y)\mathrm{d}\nu(y)\, k(x,\mathrm{d}h) \\
&\geq \int_{M \setminus\{x\}} \int_{M} \frac{P f(h) - P f(x)}{P f(h)}\Big(\frac{H(h,y)}{H(x,y)} - \frac{P f(h)}{P f(x)}\Big) H(x,y)f(y)\mathrm{d}\nu(y) \, k(x,\mathrm{d}h).
\end{split}
\end{equation*}
Having a possible singularity in mind, we choose $\varepsilon>0$  arbitrary (but small enough such that $B_\varepsilon(x) \subsetneq M$) and then use linearity to observe
\begin{align*}
\int_{M \setminus B_\varepsilon(x)} &\int_{M} \frac{P f(h) - P f(x)}{P f(h)}\Big(\frac{H(h,y)}{H(x,y)} - \frac{P f(h)}{P f(x)}\Big) H(x,y)f(y)\mathrm{d}\nu(y) \, k(x,\mathrm{d}h) \\
&= \int_{M \setminus B_\varepsilon(x)}\int_{M} \frac{P f(h) - P f(x)}{P f(h)}  H(h,y) f(y)\mathrm{d}\nu(y) \, k(x,\mathrm{d}h) \\
&\quad - \int_{M \setminus B_\varepsilon(x)}\int_{M} \frac{P f(h)-P f(x)}{P f(x)}  H(x,y) f(y) \mathrm{d}\nu(y)\, k(x,\mathrm{d}h) \\
&=: I_{\varepsilon,1} - I_{\varepsilon,2}. 
\end{align*}
Further, we have
\begin{align*}
I_{\varepsilon,1} &= \int_{M \setminus B_\varepsilon(x)}  \frac{P f(h)- P f(x)}{P f(h)} \int_{M} H(h,y) f(y) \mathrm{d}\nu(y) \, k(x,\mathrm{d}h)\\
&= \int_{M \setminus B_{\varepsilon}(x)}\big(P f(h) - P f(x)\big) k(x,\mathrm{d}h)
\end{align*}
and similarly
\begin{align*}
I_{\varepsilon,2} &= \frac{1}{P f(x)}\int_{M \setminus B_\varepsilon(x)}  \big(P f(h) - P f(x)\big)\int_{M} H(x,y) f(y)\mathrm{d}\nu(y) \, k(x,\mathrm{d}h) \\
&= \int_{M \setminus B_{\varepsilon}(x)} \big(P f(h) - P f(x)\big)k(x,\mathrm{d}h),
\end{align*}
which yields that $I_{\varepsilon,1}-I_{\varepsilon,2}=0$ holds. Sending $\varepsilon \to 0$ yields the claim.
\end{proof}
\begin{remark}
Lemma \ref{lem:generalintegralestimate} is a non-local version of the inequality
\begin{equation} \label{Unglclassic}
\int_{\iR^n}\big|\nabla_x \log H(x,y)\big|^2 H(x,y) f(y)\,\mathrm{d}y\ge \big|\nabla \log Pf(x)\big|^2 Pf(x),
\end{equation}
where $Pf(x)=\int_{\iR^n}H(x,y)f(y)\mathrm{d}y$, for sufficiently regular, positive functions $H$ and $f$.
This shows, as already indicated by the chain rule \eqref{eq:fundamentalidentity}, that the expression 
$\Psi_\Upsilon(\log f)$ serves as a natural analogue of $|\nabla \log f|^2$ in the non-local setting.

In contrast to inequality \eqref{KEY}, the argument for \eqref{Unglclassic} is quite simple. In fact,
using H\"older's inequality we have
\begin{align*}
(\partial_{x_i} Pf(x))^2 & = \big(\int_{\iR^n} \partial_{x_i} H(x,y)f(y)\,\mathrm{d}y\big)^2\\
& \le \int_{\iR^n} \frac{(\partial_{x_i} H(x,y))^2}{H(x,y)}f(y) \,\mathrm{d}y \,\,
\int_{\iR^n}  H(x,y) f(y) \,\mathrm{d}y,
\end{align*}
which directly leads to \eqref{Unglclassic} by summing up and employing the classical chain rule for the gradient.
\end{remark}
We now come to our main result. For $T>0$ arbitrary, we consider  $u:[0,T) \times M \to (0,\infty)$ with $u(\cdot,x) \in C^1\big( (0,T)\big)$ for any $x \in M$ and  such that
\begin{equation}\label{eq:heatequation}
\partial_t u(t,x) = L (u(t,\cdot))(x)
\end{equation}
holds for any $(t,x)\in (0,T)\times M$. We set $u_0(\cdot)=u(0,\cdot)$ and  make the following assumptions:
\begin{itemize}
\item[(A1)] There exists a mapping $p: (0,T) \times M \times M \to (0,\infty)$  such that $p(t,x,\cdot)$ is $\mathcal{B}(M)$-measurable and the restriction $\left.p(t,\cdot,\cdot)\right|_{M\setminus \{x\} \times M}$ is $\mathcal{B}(M \setminus \{x\}) \otimes \mathcal{B}(M)$-measurable for any $(t,x) \in(0,T)\times M$. Moreover, there exists  a $\sigma$-finite  reference measure $\mu:\mathcal{B}(M)\to[0,\infty]$, such that $p$ is differentiable with respect to time and satisfies for any $(t,x)\in (0,T)\times M$ the equation $\partial_t p(t,x,y) = L (p(t,\cdot,y))(x)$ for $\mu-$a.e. $y \in M$, and, moreover, that the representation formula
\begin{equation}\label{eq:A1solutionformula}
u(t,x)= \int_M p(t,x,y) u_0(y) \mathrm{d}\mu(y)
\end{equation}
holds true at any $(t,x) \in (0,T)\times M$.
\item[(A2)] We have that 
\begin{equation*}
\partial_t u(t,x) = \int_M \partial_t p(t,x,y) u_0(y) \mathrm{d}\mu(y)
\end{equation*}
holds at any $(t,x)\in (0,T)\times M$.
\end{itemize}

Note that assumption (A2) is quite mild in the sense that it can be justified by means of the dominated convergence theorem for a large class of examples.

Combining \eqref{eq:fundamentalidentity} and \eqref{eq:heatequation}, we observe that 
\begin{equation}\label{eq:logchainrulesultion}
\partial_t \log u = \frac{\partial_t u}{u}=L (\log u) + \Psi_\Upsilon(\log u).
\end{equation}
Note that due to assumption (A1), in \eqref{eq:logchainrulesultion}, we can also replace $u$ by $p(\cdot,\cdot,y)$ for $\mu-$a.e. $y\in M$.

The above assumptions have a clear motivation from the viewpoint of stochastic processes. Let $\big(X_t\big)_{t \geq 0}$ be a Markov process on a probability space $(\Omega,\mathcal{F},\mathbb{P})$ with state space given by a locally compact and separable metric space $M$. One defines the semigroup
\begin{equation}\label{eq:semigroupMarkovprocess}
P_t f (x) = \mathbb{E}(f(X_t) | X_t = x)
\end{equation}
for suitable functions $f$ (e.g. bounded and measurable), which is then a solution to the corresponding Cauchy problem. Very often $P_t f$ is given by an integral representation as in \eqref{eq:A1solutionformula}, where the measure $\mu$ commonly plays the role of an invariant measure for the process and the kernel $p(t,x,y)$ describes the transition density for the Markov process. Transition densities play a fundamental role in several aspects of probability theory and analysis. For instance, they form the basis of the parametrix-technique to construct L\'evy-type processes, see e.g. \cite{Jac}. The case of $\alpha$-stable isotropic L\'evy processes will be discussed in Section \ref{sec:frakLaplace} below.
Of somewhat different flavour is the situation  when $M$ is a discrete set, in which case we refer to $(X_t)_{t\geq 0}$ as a continuous-time Markov chain. Then the kernel $p(t,x,y)$ coincides with the respective transition probability, i.e.\ the probability of the process being at time $t$ in the state $y$ provided that it has started in the state $x$. We illustrate this situation in Section \ref{sec:discrete} by considering the special case when the underlying graph to the Markov generator is given by a complete graph.

The fundamental estimate of Lemma \ref{lem:generalintegralestimate} allows to reduce the problem of deriving Li-Yau inequalities for \eqref{eq:heatequation} to the kernel function of assumption (A1), as the following result shows. 
\begin{theorem}\label{thm:abstractmainresult}
Let $u:[0,T)\times M \to (0,\infty)$ solve \eqref{eq:heatequation} and satisfy assumptions (A1) and (A2) such that at any $(t,x)\in(0,T)\times M$ the expressions $L (\log p(t,\cdot,y))(x)$ for $\mu-$a.e. $y \in M$ and $L(\log u(t,\cdot))(x)$ exist, respectively. If the estimate 
\begin{equation}\label{eq:heatkernelLiYau}
-L \big(\log  p(t,\cdot,y)\big)(x)\leq \varphi(t,x)
\end{equation}
holds at any $(t,x) \in (0,T)\times M$ and $\mu-$a.e. $y\in M$, where $\varphi: (0,\infty)\times M \to \R$, then the Li-Yau type inequality
\begin{equation}\label{eq:LiYauabstract}
- L (\log u(t,\cdot))(x)\leq \varphi(t,x)
\end{equation}
holds true for every $(t,x) \in (0,T)\times M$.
\end{theorem}
\begin{remark}\label{differentialHarnack}
Using the evolution equation \eqref{eq:logchainrulesultion} for $\log u$, the Li-Yau inequality \eqref{eq:LiYauabstract}
from Theorem \ref{thm:abstractmainresult} implies the differential Harnack inequality
\begin{equation} \label{DH}
\partial_t \log u(t,x)\ge \Psi_\Upsilon(\log u)(t,x)-\varphi(t,x)
\end{equation}
for all $(t,x) \in (0,T)\times M$. In the discrete setting, it has been shown that inequalities of the form \eqref{DH} imply
Harnack inequalities for positive solutions of the diffusion equation associated with the operator $L$, see \cite{DKZ}. 
To illustrate this, we will look at the example of an unweighted complete graph in Section \ref{HarnackInequ}.
In the same section, we will further consider the fractional heat equation and show that from the corresponding differential Harnack
inequality a Harnack estimate can be derived as well. 
\end{remark}
\begin{proof} The two key ingredients in the proof of Theorem \ref{thm:abstractmainresult} are Lemma \ref{lem:generalintegralestimate} and the evolution equation \eqref{eq:logchainrulesultion} for $\log u$.
 
Using \eqref{eq:logchainrulesultion}, we can rewrite \eqref{eq:LiYauabstract} equivalently as
\begin{equation}\label{eq:Abstractclaimreformulation}
u(t,x) \, \Psi_\Upsilon(\log u(t,\cdot) )(x) \leq \partial_t u(t,x) + \varphi(t,x)\, u(t,x).
\end{equation}
By the same argument, we know by (A1) that the assumption \eqref{eq:heatkernelLiYau} can be reformulated into
\begin{equation}\label{eq:kernelassumpreformulation}
p(t,x,y) \, \Psi_\Upsilon(\log p(t,\cdot,y) )(x) \leq \partial_t p(t,x,y) + \varphi(t,x)\, p(t,x,y)
\end{equation}
for $\mu-$a.e. $y \in M$.
We observe from the assumptions (A1) and (A2) that
\begin{align*}
\partial_t u(t,x) + \varphi(t,x)\,u(t,x) &= \int_M \partial_t p(t,x,y)u_0(y) \mathrm{d}\mu(y) + \varphi(t,x) \int_M p(t,x,y) u_0(y) \mathrm{d}\mu(y) \\
&= \int_M \Big( \partial_t p(t,x,y) + \varphi(t,x)p(t,x,y) \Big) u_0(y) \mathrm{d}\mu(y) \\
&\geq \int_M \Psi_\Upsilon \big( \log p(t,\cdot,y)\big)(x)\, p(t,x,y) u_0(y) \mathrm{d}\mu(y)\\
&\geq \Psi_\Upsilon \big( \log u(t,\cdot)\big)(x) \, u(t,x),
\end{align*}
where we used \eqref{eq:kernelassumpreformulation} in the second to last line and applied Lemma \ref{lem:generalintegralestimate} in the last line (with $H(x,y)=p(t,x,y)$) at $(t,x) \in (0,T)\times M$.
\end{proof}
\begin{remark}
Note that we could relax the assumptions of Theorem \ref{thm:abstractmainresult} in order to get a local version of the above result. More precisely,  given some fixed $(t,x)\in (0,T)\times M$, we only need to assume that   $L\big( \log p(t,\cdot,y)\big)(x)$ for $\mu-$a.e. $y\in M$ and $L \big( \log u(t,\cdot)\big)(x)$ exist, respectively. If then $-L\big( \log p(t,\cdot,y)\big)(x)\leq C$  holds for $\mu-$a.e. $y\in M$ and a constant $C\in \R$, we infer  that $-L\big( \log u(t,\cdot)\big)(x)\leq C$ by the same arguments as in the proof of Theorem \ref{thm:abstractmainresult}.
\end{remark}
\section{Application to the fractional heat equation}\label{sec:frakLaplace}
In this section we want to apply Theorem \ref{thm:abstractmainresult} to derive Li-Yau type inequalities for positive solutions to the fractional heat equation
\begin{equation}\label{eq:FHE}
\partial_t u + \big(- \Delta\big)^\frac{\beta}{2}u = 0
\end{equation}
on the full space $\R^d$, where $\beta \in (0,2)$.

The fractional Laplacian on $\R^d$ can be defined as 
\begin{equation*}
- \big(-\Delta)^\frac{\beta}{2}f(x) = c_{\beta,d}\,\, \mathrm{p.v.}\int_{\R^d} \frac{f(y)-f(x)}{|x-y|^{d+\beta}}\, \mathrm{d}y,
\end{equation*}
where $\beta \in (0,2)$, the normalizing constant $c_{\beta,d}$ is given by $c_{\beta, d} = \frac{2^\beta \Gamma(\frac{d+\beta}{2})}{\pi^\frac{d}{2}|\Gamma(-\frac{\beta}{2})|}$ and $f$ is a suitable function (e.g. in $C^2(\R^d) \cap L^\infty(\R^d)$). There is also an equivalent definition in the pointwise sense that reads as  
\begin{equation*}
- \big(-\Delta)^\frac{\beta}{2}f(x) = \frac{c_{\beta,d}}{2} \int_{\R^d} \frac{f(x+y)+f(x-y) - 2f(x)}{|y|^{d+\beta}}\mathrm{d}y
\end{equation*}
and has the advantage of avoiding the principal value integral.

Using classical Fourier techniques, one can express solutions to \eqref{eq:FHE} as
\begin{equation}\label{eq:solutionfractionalheatequation}
u(t,x) = \big( G^{(\beta)}(t,\cdot) \ast u_0 \big)(x) = \int_{\R^d}G^{(\beta)}(t,x-y) u_0(y) \mathrm{d}y,
\end{equation}
where we refer to $G^{\beta)}$ as the fractional heat kernel of order $\beta \in (0,2)$. We will discuss below that this representation formula in fact holds true for the class of solutions we are interested in. 

But before that, we collect some properties of the fractional heat kernels, which we will use subsequently. We refer to \cite{Gar} and the references therein for a more detailed account.

In the special case of $\beta=1$, there is an explicit formula for the fractional heat kernel available, which reads as
\begin{equation}\label{eq:betaeq1kernel}
G^{(1)}(t,x)= \frac{\Gamma\big(\frac{d+1}{2}\big)}{\pi^\frac{d+1}{2}} \frac{t}{\big(t^2 + |x|^2\big)^{\frac{d+1}{2}}}.
\end{equation}
Concerning the general situation, it goes back to \cite{BlGe} that the fractional heat kernels behave as follows
\begin{equation}\label{eq:behavior}
G^{(\beta)}(t,x) \asymp \frac{t}{\big(t^{\frac{2}{\beta}} + |x|^2\big)^{\frac{d+\beta}{2}}}, \quad (t,x) \in (0,\infty)\times \R^d,
\end{equation}
where the symbol $\asymp$ means that the ratio is bounded by a constant positive factor from above and below.
This yields  in particular that we have for any non-negative $u_0$ 
\begin{equation}\label{eq:asympsolutionintegral}
\int_{\R^d}G^{(\beta)}(t,x-y) u_0(y) \mathrm{d}y \asymp t \int_{\R^d} \frac{u_0(y)}{\big(t^{\frac{2}{\beta}} + |x-y|^2\big)^{\frac{d+\beta}{2}}}\mathrm{d}y,
\end{equation}
which yields that the integral on the right-hand side of \eqref{eq:solutionfractionalheatequation} exists if and only if the integral on the right-hand side of \eqref{eq:asympsolutionintegral} is finite.

Further, we have $G^{(\beta)}(t,\cdot) \in C^\infty(\R^d)$ and that the fractional heat kernels can be expressed in the following self-similar form
\begin{equation}\label{eq:selfsimilarform}
G^{(\beta)}(t,x) = t^{-\frac{d}{\beta}}    \Phi_{\beta}\big(\frac{x}{t^{\frac{1}{\beta}}}\big),
\end{equation}
where $(t,x) \in (0,\infty)\times \R^d$ and $\Phi_{\beta}(y)=G^{(\beta)}(1,y)$ for any $y \in \R^d$.

There also exist specific estimates for the derivatives of the fractional heat kernels, for which we refer to \cite{KimLim}, where those have been studied extensively in quite great generality. We will only make use of 
\begin{equation}\label{eq:FSgradient}
|\nabla     \Phi_{\beta}(x)| \lesssim \frac{1}{|x|^{d+\beta+1}}
\end{equation}
and
\begin{equation}\label{eq:FSHessian}
\Vert \nabla^2     \Phi_{\beta}(x) \Vert \lesssim \frac{1}{|x|^{d+\beta+2}},
\end{equation}
which hold whenever $|x|\geq 1$, cf. \cite{KimLim}. Here we have used the symbol '$\lesssim$' to indicate that the corresponding constant is independent of $x$.

In \cite{BPSV} a Widder-type theorem for the fractional heat equation has been established. More precisely, any non-negative strong solution to \eqref{eq:FHE} is given by \eqref{eq:solutionfractionalheatequation}. Here, by strong solution it is meant that for any fixed $T>0$ we have $\partial_t u \in C\big( (0,T)\times \R^d\big)$, $u \in C \big( [0,T)\times \R^d\big)$ and that \eqref{eq:FHE} holds in the pointwise sense for every $(t,x)\in (0,T)\times \R^d$. Consequently, the assumption (A1) is satisfied for any positive strong solution $u$, where the kernel is given by $p(t,x,y)=G^{(\beta)}(t,x-y)$ and the measure $\mu$ is the Lebesgue measure on $\R^d$. We point out that it is not clear at all whether a Widder-type theorem holds for the heat equation associated with the operator $L$ in the general framework considered in Section \ref{sec:abstractsection}.

Concerning assumption (A2), we refer to \cite{VPQR}, where the bound
\begin{equation*}
| \big( - \Delta \big)^\frac{\beta}{2} (G^{(\beta)}(t,\cdot))(x)| \leq \frac{C}{\big(t^\frac{2}{\beta} + |x|^2\big)^\frac{d+\beta}{2}}
\end{equation*}
has been established for some constant $C>0$. This yields, that assumption (A2) follows from the dominated convergence theorem due to the existence of the integral on the right-hand side of \eqref{eq:solutionfractionalheatequation} (cf. \eqref{eq:asympsolutionintegral}).
\begin{lemma}\label{lem:frakheatkernelLiYau}
We have that
\begin{equation}\label{eq:fractionalheatkernelLiYau}
\big(-\Delta \big)^{\frac{\beta}{2}}(\log G^{(\beta)}(t,\cdot))(x) \leq \frac{C_{LY}(\beta,d)}{t}, 
\end{equation}
holds at any $x \in \R^d$, $t>0$ and $\beta \in (0,2)$, where the finite constant $C_{LY}(\beta,d)>0$ is given by 
\begin{equation}\label{eq:theconstant}
C_{LY}(\beta,d)=\frac{c_{\beta,d}}{2} \sup\limits_{y \in \R^d}\int_{\R^d}\frac{\log \Big(\frac{     \Phi_{\beta}(y)^2}{     \Phi_{\beta}(y+\sigma)     \Phi_{\beta}(y-\sigma)}\Big)}{|\sigma|^{d+\beta}}\mathrm{d}\sigma.
\end{equation}
\end{lemma}
\begin{proof}
We consider
\begin{align*}
\big( - \Delta\big)^{\frac{\beta}{2}}(\log G^{(\beta)}(t,\cdot))(x) &= \frac{c_{\beta,d}}{2} \int_{\R^d} \frac{2 \log G^{(\beta)}(t,x) - \log G^{(\beta)}(t,x+h)-\log G^{(\beta)}(t,x-h)}{|h|^{d+\beta}}\mathrm{d}h \\
&= \frac{c_{\beta,d}}{2} \int_{\R^d} \frac{\log \Big(\frac{G^{(\beta)}(t,x)^2}{G^{(\beta)}(t,x+h)\,G^{(\beta)}(t,x-h)}\Big)}{|h|^{d+\beta}}\mathrm{d}h.
\end{align*}
Using the self-similar form \eqref{eq:selfsimilarform} and setting $y=\frac{x}{t^{\frac{1}{\beta}}}$, we get that
\begin{align*}
 \int_{\R^d} \frac{\log \Big(\frac{G^{(\beta)}(t,x)^2}{G^{(\beta)}(t,x+h)G^{(\beta)}(t,x-h)}\Big)}{|h|^{d+\beta}}\mathrm{d}h 
&= \frac{1}{t} \int_{\R^d}\frac{\log \Big(\frac{     \Phi_{\beta}(y)^2}{     \Phi_{\beta}(y+\sigma)     \Phi_{\beta}(y-\sigma)}\Big)}{|\sigma|^{d+\beta}}\mathrm{d}\sigma,
\end{align*}
where we have substituted $\sigma=\frac{h}{t^{\frac{1}{\beta}}}$. Consequently, it suffices to show that the mapping $J:\R^d\to \R$ given by
\begin{equation*}
J(y)=\int_{\R^d}\frac{\log \Big(\frac{     \Phi_{\beta}(y)^2}{     \Phi_{\beta}(y+\sigma)     \Phi_{\beta}(y-\sigma)}\Big)}{|\sigma|^{d+\beta}}\mathrm{d}\sigma
\end{equation*}
is bounded. By continuity we only need to show that $J$ is bounded on $\R^d \setminus B_{R}(0)$ for some $R>2$. In the sequel, we assume that $|y|\geq R$. We write
\begin{align*}
J(y)&= \int_{\R^d \setminus B_1(0)}\frac{\log \Big(\frac{     \Phi_{\beta}(y)^2}{     \Phi_{\beta}(y+\sigma)     \Phi_{\beta}(y-\sigma)}\Big)}{|\sigma|^{d+\beta}}\mathrm{d}\sigma + \int_{B_1(0)} \frac{\log \Big(\frac{     \Phi_{\beta}(y)^2}{     \Phi_{\beta}(y+\sigma)     \Phi_{\beta}(y-\sigma)}\Big)}{|\sigma|^{d+\beta}}\mathrm{d}\sigma =: J_1(y)+J_2(y).
\end{align*}
First, we consider $J_1(y)$ and use the behaviour described by \eqref{eq:behavior}(with $t=1$) to find a constant $M>1$ such that
\begin{align*}
J_1(y) \leq \frac{d+\beta}{2}\,\int_{\R^d \setminus B_1(0)} \frac{\log \Big(M\frac{(1+|y+\sigma|^2)(1+|y-\sigma|^2)}{(1+|y|^2)^2}\Big)}{|\sigma|^{d+\beta}}\mathrm{d}\sigma.
\end{align*}
We have that
\begin{align*}
(1+|y+\sigma|^2)(1+|y-\sigma|^2) 
&\leq \big(1+2|y|^2 + 2|\sigma|^2)^2.
\end{align*}
In the case of $|y|> |\sigma|$, this yields that
\begin{equation*}
\frac{(1+|y+\sigma|^2)(1+|y-\sigma|^2)}{(1+|y|^2)^2} \leq \Big(\frac{1+4|y|^2}{1+|y|^2}\Big)^2 \leq M',
\end{equation*}
where $M'>1$ is some constant that is independent of $y$.
In the other case of $|y|\leq |\sigma|$, we have
\begin{equation*}
\frac{(1+|y+\sigma|^2)(1+|y-\sigma|^2)}{(1+|y|^2)^2}\leq \big(1+ 4|\sigma|^2\big)^2.
\end{equation*}
Putting these estimates together and using the monotonicity of the logarithm, we end up with
\begin{align*}
J_1(y) &\leq \frac{d+\beta}{2}\,\int_{B_{|y|}(0)\setminus B_1(0)} \frac{\log (M M')}{|\sigma|^{d+\beta}}\mathrm{d}\sigma + \frac{d+\beta}{2}\,\int_{\R^d \setminus B_{|y|}(0)} \frac{\log (M (1+4|\sigma|^2)^2)}{|\sigma|^{d+\beta}}\mathrm{d}\sigma \\
&\leq  \frac{d+\beta}{2}\,\int_{\R^d \setminus B_1(0)} \frac{\log \big(M^2 M' (1+4|\sigma|^2)^2\big)}{|\sigma|^{d+\beta}}\mathrm{d}\sigma \leq C_1 <\infty,
\end{align*}
where $C_1>0$ is some constant which is independent of $y$.

Now, we turn to $J_2(y)$. Using the inequality $\log r \leq r-1$, $r\in(0,\infty)$, we infer
\begin{equation}\label{eq:J2estimate}
J_2(y) \leq \int_{B_1(0)} \frac{     \Phi_{\beta}(y)^2 -      \Phi_{\beta}(y+\sigma)     \Phi_{\beta}(y-\sigma)}{     \Phi_{\beta}(y+\sigma)     \Phi_{\beta}(y-\sigma)}\frac{\mathrm{d}\sigma}{|\sigma|^{d+\beta}}.
\end{equation}
Given some $\eta \in B_1(0)$, we have by the Cauchy-Schwarz inequality that
\begin{align*}
1-\frac{2}{R}\leq 1- \frac{2|\eta|}{|y|} \leq 1- \frac{2|\eta|}{|y|}+\frac{|\eta|^2}{|y|^2}\leq  \frac{|y+\eta|^2}{|y|^2} \leq 1+ \frac{2|\eta|}{|y|}+\frac{|\eta|^2}{|y|^2}\leq 1+\frac{2}{R}+\frac{1}{R^2}.
\end{align*}
Consequently, since $R>2$, we observe that
\begin{equation}\label{eq:normofyestimate}
|y+\eta|^2 \asymp |y|^2,
\end{equation}
where the respective constants are in particular independent of $\eta$.
  
Further, by using Taylor's expansion we get that
\begin{align*}
     \Phi_{\beta}(y+\sigma) &=      \Phi_{\beta}(y) +  \sigma  \cdot \nabla     \Phi_{\beta}(y)^T + \frac{1}{2} \sigma \cdot \nabla^2      \Phi_{\beta}(y+\zeta_+ \sigma) \sigma , \\
     \Phi_{\beta}(y-\sigma)&=     \Phi_{\beta}(y) -  \sigma \cdot \nabla      \Phi_{\beta}(y)^T + \frac{1}{2} \sigma \cdot \nabla^2      \Phi_{\beta}(y-\zeta_- \sigma) \sigma ,
\end{align*}
where $\zeta_-,\zeta_+ \in [0,1]$. Then we get 
\begin{align*}
     \Phi_{\beta}(y)^2 -      \Phi_{\beta}(y+\sigma)     \Phi_{\beta}(y-\sigma) &= -\frac{     \Phi_{\beta}(y)}{2}\Big(\sigma \cdot \nabla^2      \Phi_{\beta}(y+\zeta_+ \sigma)\sigma+\sigma \cdot \nabla^2      \Phi_{\beta}(y-\zeta_- \sigma)\sigma\Big)\\
&\quad-\frac{\sigma \cdot \nabla      \Phi_{\beta}(y)^T}{2}\Big(\sigma \cdot \nabla^2      \Phi_{\beta}(y-\zeta_- \sigma)\sigma - \sigma \cdot \nabla^2      \Phi_{\beta}(y+\zeta_+ \sigma)\sigma \Big)\\
&\quad - \frac{1}{4}\big(\sigma \cdot  \nabla^2      \Phi_{\beta}(y+\zeta_+ \sigma)\sigma\big)\big(\sigma \cdot \nabla^2      \Phi_{\beta}(y-\zeta_-\sigma)\sigma\big)\\
&\quad + \big( \sigma \cdot \nabla      \Phi_{\beta}(y)^T \big)^2.
\end{align*}

Next, we aim to suitably estimate the integrand on the right hand side of \eqref{eq:J2estimate}. For that purpose, we will subsequently make use of the estimates {\eqref{eq:behavior} (with $t=1$),\eqref{eq:FSgradient}, \eqref{eq:FSHessian}}, \eqref{eq:normofyestimate} and the Cauchy-Schwarz inequality. Moreover, we will use the symbol `$\lesssim$' whenever the corresponding constant is independent of $|y|$. 

We have
\begin{align*}
\frac{    | \Phi_{\beta}(y)\, \sigma \cdot \nabla^2      \Phi_{\beta}(y+\zeta_+ \sigma)\sigma|}{    \Phi_{\beta}(y+\sigma)    \Phi_{\beta}(y-\sigma)}&\lesssim \Big( \frac{(1+|y+\sigma|^2)(1+|y-\sigma|^2)}{1+|y|^2}\Big)^{\frac{d+\beta}{2}} |\sigma|^2 \Vert \nabla^2     \Phi_{\beta}(y+\zeta_+\sigma) \Vert \\
&\lesssim \Big( \frac{|y+\sigma|^2|y-\sigma|^2}{|y|^2}\Big)^{\frac{d+\beta}{2}}\frac{|\sigma|^2}{|y+\zeta_+\sigma|^{d+\beta+2}} \lesssim \frac{|\sigma|^2}{|y|^2}.
\end{align*}
Clearly, the corresponding term with $-\zeta_-$ instead of $\zeta_+$ can be treated analogously. Next, we consider
\begin{align*}
&\frac{\big|\big( \sigma
\cdot \nabla     \Phi_{\beta}(y)^T\big)\big( \sigma \cdot \nabla^2     \Phi_{\beta}(y+\zeta_+\sigma)\sigma \big)\big|}{    \Phi_{\beta}(y+\sigma)    \Phi_{\beta}(y-\sigma)}\\
&\qquad\qquad\qquad\qquad\lesssim |\sigma|^3 |\nabla     \Phi_{\beta}(y)| \Vert \nabla^2     \Phi_{\beta}(y+\zeta_+\sigma)\Vert \big( (1+|y+\sigma|^2)(1+|y-\sigma|^2)\big)^{\frac{d+\beta}{2}}\\
&\qquad\qquad\qquad\qquad\lesssim \frac{|\sigma|^3 |y+\sigma|^{d+\beta}|y-\sigma|^{d+\beta}}{|y|^{d+\beta+1}|y+\zeta_+\sigma|^{d+\beta+2}} \lesssim \frac{|\sigma|^3}{|y|^3}
\end{align*}
and again, the corresponding term with $-\zeta_-$ instead of $\zeta_+$ can be treated analogously.
Furthermore, we have
\begin{align*}
&\frac{\big|\big(\sigma \cdot  \nabla^2     \Phi_{\beta}(y+\zeta_+ \sigma)\sigma\big)\big(\sigma \cdot \nabla^2     \Phi_{\beta}(y-\zeta_-\sigma)\sigma\big)\big|}{    \Phi_{\beta}(y+\sigma)    \Phi_{\beta}(y-\sigma)}\\
&\qquad\qquad\qquad\qquad\lesssim |\sigma|^4 \Vert \nabla^2     \Phi_{\beta}(y+\zeta_+\sigma)\Vert \Vert   \nabla^2  \Phi_{\beta}(y-\zeta_- \sigma)\Vert \big( (1+|y+\sigma|^2)(1+|y-\sigma|^2)\big)^{\frac{d+\beta}{2}}\\
&\qquad\qquad\qquad\qquad\lesssim \frac{|\sigma|^4 |y+\sigma|^{d+\beta} |y-\sigma|^{d+\beta}}{|y+\zeta_+\sigma|^{d+\beta+2} |y-\zeta_-\sigma|^{d+\beta+2}} \lesssim \frac{|\sigma|^4}{|y|^4}
\end{align*}
and last but not least
\begin{align*}
\frac{\big( \sigma \cdot \nabla     \Phi_{\beta}(y)^T \big)^2}{    \Phi_{\beta}(y+\sigma)    \Phi_{\beta}(y-\sigma)}&\lesssim |\sigma|^2 |\nabla     \Phi_{\beta}(y)|^2 \big( (1+|y+\sigma|^2)(1+|y-\sigma|^2)\big)^{\frac{d+\beta}{2}}\\
&\lesssim \frac{|\sigma|^2 |y+\sigma|^{d+\beta}|y-\sigma|^{d+\beta}}{|y|^{2d+2\beta+2}} \lesssim \frac{|\sigma|^2}{|y|^2}. 
\end{align*}
Putting everything together, we infer from \eqref{eq:J2estimate} that
\begin{equation*}
J_2(y) \lesssim \frac{1}{|y|^2} \int_{B_1(0)}|\sigma|^{2-d-\beta}\mathrm{d}\sigma \leq C_2 <\infty,
\end{equation*}
for some constant $C_2>0$ being independent of $y$.
\end{proof}

Observe that the proof of Lemma \ref{lem:frakheatkernelLiYau} also shows that $C_{LY}(\beta,d)$ is the smallest constant
among all $C>0$ satisfying 
\[
\big(-\Delta \big)^{\frac{\beta}{2}}(\log G^{(\beta)}(t,\cdot))(x) \leq \frac{C}{t},\quad t>0,\,x\in \iR^d.
\] 
In this sense the constant $C_{LY}(\beta,d)$ is optimal.

Combining Theorem \ref{thm:abstractmainresult} with Lemma \ref{lem:frakheatkernelLiYau}, we have shown the following Li-Yau type inequality for the fractional heat equation.
\begin{theorem}\label{thm:LiYaufrak}
Let $u:[0,T)\times \R^d \to (0,\infty)$ be a strong solution to the fractional heat equation \eqref{eq:FHE}. Then the Li-Yau type inequality
\begin{equation}\label{eq:fractionalLiYau}
\big(-\Delta\big)^\frac{\beta}{2} (\log u(t,\cdot))(x) \leq \frac{C_{LY}(\beta,d)}{t}
\end{equation}
holds at any $(t,x)\in (0,T)\times \R^d$, where the constant $C_{LY}(\beta,d)>0$ is given by \eqref{eq:theconstant}.
\end{theorem}
In the special case $\beta=1$ an analytic expression in closed form is available for the fractional heat kernel, which allows to explicitly compute the constant of the Li-Yau inequality from Theorem \ref{thm:LiYaufrak}.
\begin{proposition}
For $\beta=1$ the Li-Yau constant \eqref{eq:theconstant} is given by
\begin{equation}\label{eq:constantforbetaone}
C_{LY}(1,d)= \frac{\pi \, d(d+1)\, c_{1,d}\, \omega_d}{2} = \frac{d(d+1)}{2B\big(\frac{d+1}{2},\frac{1}{2}\big)},
\end{equation}
where $\omega_d$ denotes the volume of the $d$-dimensional unit ball in $\R^d$ and $B$ the Beta function.
\end{proposition}
\begin{proof}
First, we show that 
\begin{equation}\label{eq:beta1maxat0}
\sup\limits_{x \in \R^d}\int_{\R^d}\frac{\log \Big(\frac{     \Phi_{1}(x)^2}{\Phi_{1}(x+\sigma)\Phi_{1}(x-\sigma)}\Big)}{|\sigma|^{d+1}}\mathrm{d}\sigma = \int_{\R^d}\frac{\log \Big(\frac{\Phi_{1}(0)^2}{\Phi_{1}(\sigma)     \Phi_{1}(-\sigma)}\Big)}{|\sigma|^{d+1}}\mathrm{d}\sigma.
\end{equation}
Let $x\in \R^d$ be fixed and $\sigma \in \R^d$ arbitrary. Using the explicit representation \eqref{eq:betaeq1kernel}, we get that
\begin{align*}
\log \Big(\frac{\Phi_{1}(x)^2}{\Phi_{1}(x+\sigma)     \Phi_{1}(x-\sigma)}\Big) = \frac{d+1}{2}\log\Big(\frac{(1+|x+\sigma|^2)(1+|x-\sigma|^2)}{(1+|x|^2)^2}\Big).
\end{align*}
We have
\begin{align*}
\frac{(1+|x+\sigma|^2)(1+|x-\sigma|^2)}{(1+|x|^2)^2} &= 1 + \frac{2|\sigma|^2 + |\sigma|^4 + 2 |x|^2|\sigma|^2 - 4 (x\cdot \sigma)^2}{1+2|x|^2+|x|^4}\\
&\leq 1 + \frac{2|\sigma|^2 + |\sigma|^4 + 2 |x|^2|\sigma|^2}{1+2|x|^2+|x|^4}.
\end{align*}
We set $a:=|\sigma|^2$ and consider $h_a:[0,\infty)\to [0,\infty)$ given by $h_a(z)= \frac{2az+2a+a^2}{1+2z+z^2}$. Note that  
\begin{align*}
\big(1+h_a(0)\big)^{\frac{d+1}{2}} = \big(1+2|\sigma|^2+|\sigma|^4\big)^\frac{d+1}{2} = \frac{\Phi_1(0)^2}{\Phi_1(\sigma)\Phi_1(-\sigma)}.
\end{align*}
Hence, in order to establish \eqref{eq:beta1maxat0} it suffices to show that $h_a$ is decreasing on $[0,\infty)$ for any $a\in [0,\infty)$. We have that 
\begin{align*}
h_a'(z) = \frac{(1+2z+z^2)2a - (2az+2a+a^2)(2z+2)}{(1+2z+z^2)^2},
\end{align*}
which is non-positive if and only if 
\begin{equation*}
\frac{1+2z+z^2}{2z+2}-z \leq 1+\frac{a}{2}.
\end{equation*}
This inequality holds true since 
\begin{align*}
\frac{1+2z+z^2}{2z+2}-z = \frac{1-z^2}{2z+2} = \frac{1-z}{2}\leq \frac{1}{2}.
\end{align*}
Consequently, \eqref{eq:beta1maxat0} is established. Due to $\Phi_1$ being radial symmetric, the calculation of the corresponding integral boils down to the one-dimensional case. More precisely, we have
\begin{align*}
\int_{\R^d}\frac{\log \Big(\frac{\Phi_{1}(0)^2}{\Phi_{1}(\sigma)     \Phi_{1}(-\sigma)}\Big)}{|\sigma|^{d+1}}\mathrm{d}\sigma = 2 \,\frac{d+1}{2}\, d\, \omega_d \int_0^\infty \frac{\log \big( 1+r^2)}{r^2}\mathrm{d}r = \pi\, d(d+1) \, \omega_d,
\end{align*}
where the latter follows from an elementary calculation. The right-hand side in \eqref{eq:constantforbetaone} follows from $\omega_d=\frac{\pi^\frac{d}{2}}{\Gamma(\frac{d}{2}+1)}$ and $\Gamma(\frac{1}{2})=\sqrt{\pi}$, which yields
\begin{align*}
\pi  c_{1,d} \omega_d =  \sqrt{\pi} \frac{\Gamma(\frac{d+1}{2})}{\Gamma(\frac{d}{2}+1)} =  \frac{\Gamma(\frac{1}{2})\Gamma(\frac{d+1}{2})}{\Gamma(\frac{d}{2}+1)} = \frac{1}{B\big(\frac{d+1}{2},\frac{1}{2}\big)}.
\end{align*}
\end{proof}

\begin{remark}
It is still an open question whether the Li-Yau constant $C_{LY}(\beta,d)$ tends to the Li-Yau constant $\frac{d}{2}$ from the
classical heat equation as $\beta\to 2$.
\end{remark}

As highlighted throughout Section \ref{sec:abstractsection} our approach merely depends on the specific structure of the operator than on the state space. Therefore it is natural to ask whether fractional powers of the Laplace-Beltrami operator on a Riemannian
manifold do fit into our setting. The question of suitable representations of fractional powers of the Laplace-Beltrami operator has been addressed in the context of compact
manifolds in \cite{ACM} and also for hyperbolic spaces in \cite{BGS}. These formulas
are based on well-established ways to represent fractional powers of suitable operators like the Laplace-Beltrami operator,
e.g.\ via the heat semigroup and Gaussian estimates (\cite{ACM}) or the harmonic extension method (\cite{BGS}).
In the discrete setting, it has been shown by \cite{CRSTV} that fractional powers of the discrete Laplacian with state space $\iZ$ can be written in the form described in Section \ref{sec:abstractsection}. However, it seems to be difficult to establish
suitable bounds on the heat kernel in a more general Riemannian setting that allow to derive a Li-Yau inequality (with an explicit
function $\varphi(t,x)$) via the reduction principle, Theorem \ref{thm:abstractmainresult}.

Another interesting research direction is gradient estimates for more complicated equations, which 
may comprise non-local terms like a fractional Laplacian (not necessarily as the leading order operator). 
We refer to the recent work \cite{CDFGV}, where it could be also interesting to include a fractional Laplacian
in the setting described in lines 2-9 on page 437.
\section{Illustration of the discrete case}\label{sec:discrete}
Let us now discuss an application of the general principle developed in Section \ref{sec:abstractsection} that is very different from the one of the last section. In the case that $M$ is a finite set, the operator $L$ can be viewed as the generator of a continuous-time Markov chain.  More precisely, $L$ induces the generator matrix (which is commonly called $Q$-matrix) given by $Q=\big(q(x,y)\big)_{x,y \in M}$ where $q(x,\cdot)$ denotes the density of the measure $k(x,\cdot)$ with respect to the counting measure on $M\setminus \{x\}$. Further, one sets $q(x,x) = -\sum_{y \neq x} q(x,y)$. In this setting, the kernel $p(t,x,y)$ in assumption (A1) is given by the transition probabilities and assumption (A2) is obviously satisfied. The transition probabilities are given by the entries of the matrix exponential $e^{tQ}$. This, and much more information on the theory of continuous-time Markov chains can be found, for instance, in \cite{Nor}. 

We illustrate the finite state space case by the example of the (unweighted) complete graph $K_n$. In this case, the corresponding $Q$-matrix is given by the entries ($x,y\in \{1,2,\ldots,n\}$)
\begin{equation}
q(x,y)= \left\{
\begin{array}{l@{\;:\;}l}
1 & x\neq y\\
-(n-1) & x=y.
\end{array}\right .
\end{equation}
One readily checks that $0$ is an eigenvalue of $Q$ of multiplicity $1$ with eigenvector $(1,...,1)^T$ and $-n$ is an eigenvalue of $Q$ of multiplicity $n-1$ with eigenvectors $(-1,e_j^{(n-1)})^T$, $j \in \{1,...,n-1\}$, where $e_j^{(n-1)}$ denotes the $j$-th unit vector in $\R^{n-1}$.
From that, one determines the transition probabilities to be given by
\begin{equation*}
p(t,x,y) = \left\{\begin{array}{l@{\;:\;}l}
\frac{1-e^{-nt}}{n} & x\neq y, \\
\frac{1+(n-1)e^{-nt}}{n} & x=y.
\end{array}\right.
\end{equation*}
Now, we have
\begin{equation*}
- L \big(\log p(t,\cdot,y)\big)(x) = \sum_{z \neq x} \log \frac{p(t,x,y)}{p(t,z,y)} = \left\{\begin{array}{l@{\;:\;}l}
(n-1) \log \big( \frac{1+(n-1)e^{-nt}}{1-e^{-nt}}\big) & x=y\\
\log \big(\frac{1-e^{-nt}}{1+(n-1)e^{-nt}}\big) & x\neq y,
\end{array}\right.
\end{equation*}
which yields that $-L \big(\log (p)(t,\cdot,y)\big)(x)$ is maximal if $x=y$ for any fixed $t>0$. 

Applying Theorem \ref{thm:abstractmainresult}, we obtain the following optimal result.
\begin{theorem}\label{thm:KnLiYau}
Let $u:[0,\infty)\times M \to (0,\infty)$ be a solution to $\partial_t u = L u$ on $(0,\infty)\times M$, where $L$ denotes  the generator of the continuous-time Markov chain such that  the complete graph $K_n$, $n\geq 2$, is the corresponding underlying graph. Then the Li-Yau inequality
\begin{equation*}
- L \big( \log u(t,\cdot))(x)\leq (n-1) \log \big(\frac{1+(n-1)e^{-nt}}{1-e^{-nt}}\big)
\end{equation*}
holds at any $(t,x) \in (0,\infty)\times M$.
\end{theorem}
Let us compare this result with what has been obtained in \cite{DKZ} in the case of the complete graph $K_n$. The method of \cite{DKZ} relies on the curvature-dimension condition $CD(F;0)$ (where $F$ is a so called $CD$-function). It is shown there, that $CD(F;0)$ implies a Li-Yau inequality of the form 
\begin{equation*}
- L \big( \log u(t,\cdot)\big)(x) \leq \varphi(t),
\end{equation*}
where the so-called relaxation function $\varphi:(0,\infty)\to(0,\infty)$ is defined as the unique positive solution of 
\begin{equation}\label{eq:relaxationsfunction}
\varphi'(t) + F(\varphi(t)) = 0,\quad t>0,
\end{equation}
such that $\lim\limits_{t\to 0^+}\varphi(t)=\infty$. It is important to notice that for the uniqueness part, there is no need to assume that $\frac{F(r)}{r}$ is increasing on $(0,\infty)$. This can be seen from the proof of \cite[Lemma 3.5]{DKZ}.

It has been calculated in \cite{DKZ} that the function $F(r)=(n-1)\big(e^\frac{r}{n-1}-(n-1)e^\frac{-r}{n-1}+n-2\big)
$, $r\geq 0$, satisfies the functional inequality that defines the curvature-dimension condition but is not a $CD$-function in general since $\frac{F(r)}{r}$ is not increasing near $0$. Since this property is of eminent importance for the method of \cite{DKZ}, one needs to find a suitable lower bound for $F$ which serves as a $CD$-function. This is when optimality in general gets lost. Remarkably, Theorem \ref{thm:KnLiYau} shows that in this situation this is only an obstacle due to the method of \cite{DKZ}, in fact the function $F$ is tailor-made for the Li-Yau inequality of Theorem \ref{thm:KnLiYau}. Indeed, we can show that 
\begin{equation*}
\varphi(t)=(n-1) \log\big(\frac{1+(n-1)e^{-nt}}{1-e^{-nt}}\big)
\end{equation*}
is the relaxation function to $F$, in the sense that $\varphi$ is the unique positive solution of \eqref{eq:relaxationsfunction} with $\lim\limits_{t\to 0^+}\varphi(t)=\infty$.

More precisely, we have
\begin{align*}
\varphi'(t) &= \frac{(n-1)n e^{-nt}}{(1+(n-1)e^{-nt})(1-e^{-nt})} \big( (1-n) (1- e^{-nt})-(1+(n-1)e^{-nt})\big) \\
&= - \frac{n^2(n-1)e^{-nt}}{(1+(n-1)e^{-nt})(1-e^{-nt})}
\end{align*}
and besides that
\begin{align*}
&F(\varphi(t)) = (n-1) \Big( \frac{1+(n-1)e^{-nt}}{1-e^{-nt}} - (n-1) \frac{1-e^{-nt}}{1+(n-1)e^{-nt}} + n-2 \Big) \\
&= \frac{(n-1)  \Big( (1+(n-1)e^{-nt})^2 - (n-1)(1-e^{-nt})^2 + (n-2)(1-e^{-nt})(1+(n-1)e^{-nt})\Big)}{(1+(n-1)e^{-nt})(1-e^{-nt})} \\
&= \frac{(n-1)\Big( 2-n + 4(n-1) e^{-nt} - (n-1)(n-2)e^{-2nt} + (n-2)\big( 1+(n-2)e^{-nt} - (n-1)e^{-2nt}\big)\Big)}{(1+(n-1)e^{-nt})(1-e^{-nt})}\\
&= \frac{n^2 (n-1)e^{-nt}}{(1+(n-1)e^{-nt})(1-e^{-nt})}.
\end{align*}

\section{Harnack inequalities} \label{HarnackInequ}
One of the main reasons for which Li and Yau \cite{LY} proved their inequality was to deduce from it their fundamental
Harnack inequality. In the parabolic case, a Harnack inequality is a pointwise estimate of the form
\begin{equation}\label{eq:Harnack}
u(t_1,x_1)\leq u(t_2,x_2) C(t_1,t_2,x_1,x_2),\quad 0<t_1<t_2<\infty,\, x_1,x_2\in M,
\end{equation}
where $u$ is a positive solution to the corresponding diffusion/heat equation on the state space $M$, and the constant in 
\eqref{eq:Harnack} may also depend on parameters of the space $M$ like dimension and curvature bounds.
For example, positive solutions of the heat equation on $\iR^d$
satisfy the following sharp inequality
\begin{equation} \label{Har1}
u(t_1,x_1)\le u(t_2,x_2) \Big(\frac{t_2}{t_1}\Big)^{d/2} e^{\frac{|x_1-x_2|^2}{4(t_2-t_1)}},\quad 0<t_1<t_2,
\,x_1,\,x_2\in \iR^d,
\end{equation}
see e.g.\ \cite{AB}. A corresponding inequality is also valid on complete $d$-dimensional Riemannian manifolds $M$ with nonnegative
Ricci curvature, see \cite{LY}.
 
The main idea to derive such estimates for the classical heat equation on a Riemannian manifold from the Li-Yau inequality is to integrate along geodesics in space-time and use the classical chain rule. In the context of jump processes it is natural to replace this integration along 
continuous paths with appropriate jumps, especially if respective paths in space-time are not available at all like in the discrete setting. The latter situation was studied in the works \cite{BHLLMY15}, \cite{DKZ} and \cite{Mun}, where Harnack inequalities have been derived from Li-Yau estimates. 

In order to highlight the additional difficulty that arises in the space-continuous setting, we briefly repeat the core argument from the discrete case of how Li-Yau estimates can be used in order to derive Harnack estimates of the form \eqref{eq:Harnack}. Recall Remark \ref{differentialHarnack}, which states that the non-local Li-Yau inequality $-L \log u \leq \varphi(t)$ can be reformulated as
\[
\partial_t \log u \geq \Psi_\Upsilon(\log u) - \varphi(t).
\]
With this at hand, one proceeds as follows. Given $0<t_1<t_2$ one has for arbitrary $s\in [t_1,t_2]$
\begin{align*}
\log \frac{u(t_1,x_1)}{u(t_2,x_2)} &= \log \frac{u(t_1,x_1)}{u(s,x_1)} + \log \frac{u(s,x_1)}{u(s,x_2)} + \log \frac{u(s,x_2)}{u(t_2,x_2)}\\
&= - \int_{t_1}^s \partial_t \log u(t,x_1) \mathrm{d}t +
\log \frac{u(s,x_1)}{u(s,x_2)} - \int_{s}^{t_2} \partial_t \log u(t,x_2) \mathrm{d}t\\
&\leq \int_{t_1}^s \big( \varphi(t)-\Psi_{\Upsilon}(\log  u)(t,x_1)\big)\mathrm{d}t + \log \frac{u(s,x_1)}{u(s,x_2)} + \int_s^{t_2} \big(\varphi(t)-\Psi_\Upsilon(\log u)(t,x_2)\big)\mathrm{d}t\\
&= \int_{t_1}^{t_2} \varphi(t)\mathrm{d}t + \log \frac{u(s,x_1)}{u(s,x_2)} - \int_{t_1}^s \Psi_\Upsilon (\log u) (t,x_1)\mathrm{d}t - \int_{s}^{t_2} \Psi_\Upsilon (\log u) (t,x_2) \mathrm{d}t.
\end{align*}
Note that it suffices to find a control
\[
\log \frac{u(s,x_1)}{u(s,x_2)} - \int_{t_1}^s \Psi_\Upsilon (\log u) (t,x_1)\mathrm{d}t - \int_{s}^{t_2} \Psi_\Upsilon (\log u) (t,x_2) \mathrm{d}t \leq \tilde{C}(t_1,t_2,x_1,x_2)
\]
in order to derive a bound of the form \eqref{eq:Harnack}. In the discrete setting this can be achieved with a quite rough estimate. Indeed, let $M$ be discrete and assume that $k(x_2,x_1)>0$, then we observe
\begin{align}
&\log \frac{u(s,x_1)}{u(s,x_2)} - \int_{t_1}^s \Psi_\Upsilon (\log u) (t,x_1)\mathrm{d}t - \int_{s}^{t_2} \Psi_\Upsilon (\log u) (t,x_2) \mathrm{d}t \nonumber\\
&\qquad\qquad\leq \log \frac{u(s,x_1)}{u(s,x_2)} - k(x_2,x_1) \int_s^{t_2} \Upsilon(\log u(t,x_1)-\log u(t,x_2))\mathrm{d}t,
\label{controlneeded}
\end{align}
from which a suitable estimate can then be derived, see \cite[Section 6]{DKZ}. In the argument, one chooses, among other things, $s\in[t_1,t_2]$
in such a way that the right-hand side in \eqref{controlneeded} is minimized over $[t_1,t_2]$, and one uses the inequality $\Upsilon(z)\ge \frac{1}{2}z^2$,
$z\ge 0$.

Combining the proof from \cite[Section 6]{DKZ} with our findings from Section \ref{sec:discrete} we arrive at the following Harnack inequality for positive solutions of the heat equation on complete graphs.
\begin{corollary}
Let $u:[0,\infty)\times M \to (0,\infty)$ be a solution to $\partial_t u= Lu$ on $(0,\infty)\times M$, where $L$ denotes the generator of the continuous-time Markov chain such that the complete graph $K_n$, $n\geq 2$, is the corresponding underlying graph. Then the Harnack inequality 
\begin{equation}
u(t_1,x_1)\leq u(t_2,x_2) \exp \Big[ (n-1)\int_{t_1}^{t_2} \log\Big( \frac{1+(n-1)e^{-nt}}{1-e^{-nt}}\Big) \mathrm{d}t + \frac{2}{t_2-t_1}\Big]
\end{equation}
holds true for any $0<t_1<t_2<\infty$ and $x_1,x_2 \in M$.
\end{corollary}

As described above, in the discrete setting, it is possible to drop the $\Psi_\Upsilon(\log u)(\cdot,x_1)$ term and to only use
the corresponding second term with $x=x_2$. Estimating in this simple way does not work if the state space is continuous.
It turns out that here the proof of a Harnack inequality in the form \eqref{eq:Harnack} is much more involved. 

We now come to a Harnack inequality for positive solutions to the fractional heat equation on $\mathbb{R}^d$. Even though
the result obtained here is not optimal, the proof shows that from the Li-Yau inequality \eqref{eq:fractionalLiYau} one can 
indeed derive Harnack estimates. To the best of the authors' knowledge, this seems to be the first proof of a Harnack inequality
via a Li-Yau type inequality for an evolution equation with a purely non-local diffusion operator in the space-continuous setting.
\begin{theorem} \label{HarnackFracEvol}
Let $\beta\in (0,2)$ and $u:[0,\infty)\times \mathbb{R}^d\to (0,\infty)$ be a strong solution to the fractional heat equation \eqref{eq:FHE}. Then there exists a positive constant $C=C(\beta,d)$ such that for all
$0<t_1<t_2<\infty$ and $x_1,x_2 \in \mathbb{R}^d$ the following Harnack type inequality holds true.
\begin{equation} \label{HarnackUnlg}
u(t_1,x_1)\le u(t_2,x_2) \Big(\frac{t_2}{t_1}\Big)^{C_{LY}}\exp \left(C\left[1+\frac{|x_1-x_2|^{\beta+d}}
{(t_2-t_1)^{1+\frac{d}{\beta}}}\right]\right),
\end{equation}
where $C_{LY}$ is the Li-Yau constant given by \eqref{eq:theconstant}.
\end{theorem}
\begin{remark}\label{ClassicalHarnack}
(i) The parabolic Harnack inequality for the space fractional heat equation has already intensively been studied in the literature.
Concerning local solutions, we refer to \cite{BL2} and \cite{CK} for various results in this direction. Here the authors make use of probabilistic methods.
A purely analytic proof, even in a rough setting, has been found in \cite{CLD}. For global solutions, Harnack estimates have
also been derived in \cite[Theorem 8.2]{BoSiVa} and \cite[Theorem 2.7]{DKSZ} by means of heat kernel estimates.
As shown in \cite{BoSiVa}, an important difference to the classical heat equation is that no time gap is required between $t_1$ and $t_2$.

(ii) The Harnack inequality \eqref{HarnackUnlg} is not optimal in various respects. We need a time delay and instead of the exponential function one would expect polynomial terms, cf.\ \cite[Theorem 8.2]{BoSiVa}. Another aspect is the lack of robustness, the constant $C$ blows up as $\beta\to 2$. Nevertheless \eqref{HarnackUnlg} is scale invariant and implies a natural Harnack estimate on different
space-time cylinders as typically obtained for related equations with rough coefficients, cf.\  \cite{CLD}.
\end{remark}
\begin{proof}
As in the previous calculations we have for any $s\in [t_1,t_2]$ that
\begin{align} \label{HarnackStart1}
\log \frac{u(t_1,x_1)}{u(t_2,x_2)} &=
 - \int_{t_1}^s \partial_t \log u(t,x_1) \mathrm{d}t +
\log \frac{u(s,x_1)}{u(s,x_2)} - \int_{s}^{t_2} \partial_t \log u(t,x_2) \mathrm{d}t.
\end{align}
From Theorem \ref{thm:LiYaufrak} and Remark \ref{differentialHarnack} we know that the differential Harnack inequality
\[
\partial_t \log u(t,x)\ge \Psi_\Upsilon(\log u)(t,x)-\frac{C_{LY}(\beta,d)}{t},\quad t>0,\,x\in \iR^d
\]
holds true. Combining this with \eqref{HarnackStart1} we obtain
\begin{align}
\log \frac{u(t_1,x_1)}{u(t_2,x_2)} & \le 
 \int_{t_1}^{t_2}\frac{C_{LY}(\beta,d)}{t} \mathrm{d}t + \log \frac{u(s,x_1)}{u(s,x_2)}\nonumber\\
 & \quad - \int_{t_1}^s \Psi_\Upsilon (\log u) (t,x_1)\mathrm{d}t - \int_{s}^{t_2} \Psi_\Upsilon (\log u) (t,x_2) \mathrm{d}t,
 \;s\in [t_1,t_2].
 \label{HarnackStart2}
\end{align}
So setting $v=\log u$, just like before, we have to find a suitable upper bound for the function 
\begin{equation*}
f(s):=v(s,x_1)-v(s,x_2) - \int_{t_1}^s \Psi_\Upsilon (v) (t,x_1)\mathrm{d}t - \int_{s}^{t_2} \Psi_\Upsilon (v) (t,x_2) \mathrm{d}t,\quad s \in [t_1,t_2], 
\end{equation*} 
We first assume that $|x_1-x_2|\le 1$ and show the general statement by a classical scaling argument later.

Set $t_*=\frac{t_1+t_2}{2}$. We introduce the weight function 
\begin{equation}
\eta(t)=\left\{\begin{array}{ll}
(t-t_1)^\alpha &, t \in [t_1,t_*)\\
(t_2-t)^\alpha &, t \in [ t_*,t_2],
\end{array}\right. 
\end{equation}
with fixed parameter $\alpha>\frac{1}{2}\max\{0,\frac{d}{\beta}-1\}$, e.g.\ $\alpha=\frac{d}{\beta}$.
Employing Fubini's theorem, we observe that
\begin{align}
\min\limits_{t \in [t_1,t_2]} f(t) &\leq \frac{1}{\int_{t_1}^{t_2} \eta(t) \mathrm{d}t} \int_{t_1}^{t_2} \eta(t)f(t)\mathrm{d}t
\nonumber\\
&= \frac{1}{\int_{t_1}^{t_2} \eta(t) \mathrm{d}t} \int_{t_1}^{t_2}\Big( \eta(t) \big(v(t,x_1)-v(t,x_2)\big) - \Psi_\Upsilon(v)(t,x_1) \int_t^{t_2}\eta(\tau)\mathrm{d}\tau \nonumber\\
&\qquad\qquad\qquad\qquad- \Psi_\Upsilon(v)(t,x_2) \int_{t_1}^{t}\eta(\tau)\mathrm{d}\tau \Big)\mathrm{d}t\nonumber\\
&= \frac{1}{\int_{t_1}^{t_2} \eta(t) \mathrm{d}t}\Big[ \int_{t_1}^{t_2}\Big( \eta(t) A_1(t) - \Psi_\Upsilon(v)(t,x_1) \int_t^{t_2}\eta(\tau)\mathrm{d}\tau\Big)\mathrm{d}t \nonumber\\
&\qquad\qquad\qquad+ \int_{t_1}^{t_2} \Big(\eta(t)A_2(t)- \Psi_\Upsilon(v)(t,x_2) \int_{t_1}^{t}\eta(\tau)\mathrm{d}\tau \Big)\mathrm{d}t\Big],\label{fmin}
\end{align}
where 
\begin{align*}
A_1(t) = \frac{1}{|Q_t|}\int_{Q_t} \big(v(t,x_1)-v(t,y)\big)\mathrm{d}y,\quad
A_2(t)  = \frac{1}{|Q_t|}\int_{Q_t}\big( v(t,y)-v(t,x_2)\big)\mathrm{d}y
\end{align*}
for $t \in (t_1,t_2)$, and $Q_t$ is the open ball $Q_t=B_{r(t)}(x_1)$ with radius
\begin{equation*}
r(t)=
\Big(\frac{\omega_d c_{\beta,d}}{1+\alpha}(t_2-t)\Big)^\frac{1}{\beta}
\end{equation*}
and volume
\begin{equation}\label{eq:Qtvolume}
|Q_t|=\omega_d r(t)^d= \omega_d^{\frac{d}{\beta}+1} \Big(\frac{c_{\beta,d}}{1+\alpha}(t_2-t)\Big)^\frac{d}{\beta}.
\end{equation}

Using the inequality $z \leq \Upsilon(-z)+1$, valid for any $z \in \mathbb{R}$, we may estimate
\begin{align*}
A_1(t)\leq 1+ \frac{1}{|Q_t|}\int_{Q_t} \Upsilon\big(v(t,y)-v(t,x_1)\big)\mathrm{d}y \leq  1+\frac{c_{\beta,d}(t_2-t)}{1+\alpha}\int_{Q_t}\frac{\Upsilon(v(t,y)-v(t,x_1))}{|y-x_1|^{d+\beta}}\mathrm{d}y.
\end{align*}
Consequently, we observe that
\begin{align*}
&\frac{1}{\int_{t_1}^{t_2} \eta(t) \mathrm{d}t} \int_{t_1}^{t_2}\Big( \eta(t) A_1(t) - \Psi_\Upsilon(v)(t,x_1) \int_t^{t_2}\eta(\tau)\mathrm{d}\tau\Big)\mathrm{d}t \\
&\qquad\qquad \leq 1 +\frac{c_{\beta,d}}{\int_{t_1}^{t_2} \eta(t) \mathrm{d}t} \int_{t_1}^{t_2}\Big( \frac{\eta(t)(t_2-t)}{1+\alpha}-\int_t^{t_2}\eta(\tau)\mathrm{d}\tau \Big) \int_{Q_t} \frac{\Upsilon(v(t,y)-v(t,x_1))}{|y-x_1|^{d+\beta}}\mathrm{d}y \,\mathrm{d}t.
\end{align*}
If $t \in [t_*, t_2 ]$ then
\begin{equation}\label{eq:factorforA1}
 \frac{\eta(t)(t_2-t)}{1+\alpha}-\int_t^{t_2}\eta(\tau)\mathrm{d}\tau  = 0.
\end{equation}
If instead $t \in [ t_1, t_*)$ one readily checks that the left-hand side of \eqref{eq:factorforA1} is increasing in $t$, so that we conclude
\begin{equation} \label{A1Estimate}
\frac{1}{\int_{t_1}^{t_2} \eta(t) \mathrm{d}t} \int_{t_1}^{t_2}\Big( \eta(t) A_1(t) - \Psi_\Upsilon(v)(t,x_1) \int_t^{t_2}\eta(\tau)\mathrm{d}\tau\Big)\mathrm{d}t\leq 1.
\end{equation}

Next, denoting by $z_+$ the positive part of $z\in \iR$ and using that $\Upsilon(z)\ge \frac{1}{2}z^2$ for all $z\ge 0$, we estimate the integral term involving $A_2(t)$ as follows.
\begin{align*}
&\int_{t_1}^{t_2}\Big( \eta(t)A_2(t)- \Psi_\Upsilon(v)(t,x_2) \int_{t_1}^{t}\eta(\tau)\mathrm{d}\tau\Big) \mathrm{d}t\\
&\qquad\qquad\leq \int_{t_1}^{t_2}\int_{Q_t} \frac{\eta(t)(v(t,y)-v(t,x_2))}{|Q_t|}-\frac{c_{\beta,d}\Upsilon(v(t,y)-v(t,x_2))\int_{t_1}^{t}\eta(\tau)\mathrm{d}\tau}{|y-x_2|^{d+\beta}}\mathrm{d}y\mathrm{d}t\\
&\qquad\qquad\leq \int_{t_1}^{t_2}\int_{Q_t} \frac{\eta(t)(v(t,y)-v(t,x_2))_+}{|Q_t|}-\frac{c_{\beta,d}\Upsilon\big((v(t,y)-v(t,x_2))_+\big)\int_{t_1}^{t}\eta(\tau)\mathrm{d}\tau}{|y-x_2|^{d+\beta}} \mathrm{d}y\mathrm{d}t\\
&\qquad\qquad\leq \int_{t_1}^{t_2}\int_{Q_t}\frac{\eta(t)(v(t,y)-v(t,x_2))_+}{|Q_t|}-\frac{c_{\beta,d}\big((v(t,y)-v(t,x_2))_+\big)^2\int_{t_1}^{t}\eta(\tau)\mathrm{d}\tau}{2|y-x_2|^{d+\beta}} \mathrm{d}y\mathrm{d}t\\
&\qquad\qquad\leq \int_{t_1}^{t_2} \int_{Q_t} \frac{|y-x_2|^{d+\beta}\eta(t)^2}{2c_{\beta,d}|Q_t|^2 \int_{t_1}^{t}\eta(\tau)\mathrm{d}\tau}\mathrm{d}y\mathrm{d}t,
\end{align*}
where in the last step we used that $\max_{z \in \mathbb{R}}(b_1z-b_2z^2)= \frac{b_1^2}{4b_2}$ for constants $b_1 \in \mathbb{R}$ and $b_2>0$. 

Using the inequality $(a+b)^p\le 2^{p-1}(a^p+b^p)$, $a,b\ge 0$, $p\ge 1$, and our assumption that $|x_1-x_2|\leq 1$ we have
\begin{align*}
\int_{Q_t} |y-x_2|^{d+\beta}\mathrm{d}y &\leq 2^{d+\beta-1} \Big(\int_{Q_t} |y-x_1|^{d+\beta}\mathrm{d}y +  |Q_t|\Big)\\
&\le  2^{d+\beta-1} \big(r(t)^{d+\beta}|Q_t|+|Q_t|\big).
\end{align*}
We therefore obtain
\[
\int_{t_1}^{t_2}\Big( \eta(t)A_2(t)- \Psi_\Upsilon(v)(t,x_2) \int_{t_1}^{t}\eta(\tau)\mathrm{d}\tau\Big) \mathrm{d}t
\le \frac{2^{d+\beta-2}}{c_{\beta,d}}\int_{t_1}^{t_2} \frac{\eta(t)^2\big(1+r(t)^{d+\beta}\big)}{|Q_t| \int_{t_1}^{t}\eta(\tau)\mathrm{d}\tau}\,\mathrm{d}t.
\]
For $t \in [t_1, t_*]$ we observe that
\begin{align*}
 \frac{\eta(t)^2\big(1+r(t)^{d+\beta}\big)}{|Q_t| \int_{t_1}^{t}\eta(\tau)\mathrm{d}\tau}
= \frac{(t-t_1)^{2\alpha} \big(1+r(t)^{d+\beta}\big)(1+\alpha)}{\omega_d r(t)^d (t-t_1)^{1+\alpha}}
\le \frac{1+\alpha}{\omega_d} (t-t_1)^{\alpha-1}\Big(\frac{1}{r(t_*)^d}+r(t_1)^\beta\Big)
\end{align*}
whereas for $t \in (t_*, t_2]$ we have
\begin{align*}
 \frac{\eta(t)^2\big(1+r(t)^{d+\beta}\big)}{|Q_t| \int_{t_1}^{t}\eta(\tau)\mathrm{d}\tau}
& \le  \frac{(t_2-t)^{2\alpha}  \big(1+r(t)^{d+\beta}\big)}{\omega_d r(t)^d  \int_{t_1}^{t_*}\eta(\tau)\mathrm{d}\tau}
\le  \frac{1+\alpha}{\omega_d}\frac{ (t_2-t)^{2\alpha} }{(t_*-t_1)^{1+\alpha}}\Big(\frac{1}{r(t)^d}+r(t_*)^\beta\Big)\\
& =  \frac{1+\alpha}{\omega_d}\frac{ (t_2-t)^{2\alpha} }{(t_*-t_1)^{1+\alpha}}
\Big[\Big(\frac{1+\alpha}{\omega_d c_{\beta,d}}\Big)^{\frac{d}{\beta}} (t_2-t)^{-\frac{d}{\beta}}+r(t_*)^\beta\Big].
\end{align*}
Note that $2\alpha-\frac{d}{\beta}>-1$, by our choice of $\alpha$. Consequently,
\begin{align*}
\int_{t_1}^{t_2}  \frac{\eta(t)^2\big(1+ r(t)^{d+\beta}\big)}{|Q_t| \int_{t_1}^{t}\eta(\tau)\mathrm{d}\tau}\,\mathrm{d}t
&= \int_{t_1}^{t_*}\ldots\,\mathrm{d}t+
\int_{t_*}^{t_2} \ldots\,\mathrm{d}t\\
&\le \frac{(1+\alpha)(t_*-t_1)^\alpha}{\alpha \omega_d} \Big(\frac{1}{r(t_*)^d}+r(t_1)^\beta\Big)\\
& \quad +\frac{(1+\alpha)}{\omega_d} \Big(\frac{(1+\alpha)^{\frac{d}{\beta}}(t_2-t_*)^{\alpha-\frac{d}{\beta}}}
{(2\alpha-\frac{d}{\beta}+1)(\omega_d c_{\beta,d})^{\frac{d}{\beta}}}+\frac{(t_2-t_*)^\alpha}{2\alpha+1}r(t_*)^\beta
\Big)\\
&\le \frac{(1+\alpha)^{1+\frac{d}{\beta}}(t_*-t_1)^{\alpha-\frac{d}{\beta}}}{\alpha \omega_d^{1+\frac{d}{\beta}}
c_{\beta,d}^{\frac{d}{\beta}}}+\frac{2c_{\beta,d}}{\alpha}(t_*-t_1)^{1+\alpha}\\
& \quad+ \frac{(1+\alpha)^{1+\frac{d}{\beta}}(t_2-t_*)^{\alpha-\frac{d}{\beta}}}
{(2\alpha-\frac{d}{\beta}+1)\omega_d^{1+\frac{d}{\beta}} c_{\beta,d}^{\frac{d}{\beta}}}
+\frac{c_{\beta,d}}{2\alpha+1}(t_2-t_*)^{1+\alpha}.
\end{align*}
Since $t_2-t_*=t_*-t_1=\frac{1}{2}(t_2-t_1)$ and
\[
\int_{t_1}^{t_2}\eta(t)\,\mathrm{d}t=2\int_{t_1}^{t_*}\eta(t)\,\mathrm{d}t=\frac{2}{1+\alpha}(t_*-t_1)^{1+\alpha},
\]
it follows from the previous estimates that
\begin{align}
\frac{1}{\int_{t_1}^{t_2} \eta(t) \mathrm{d}t}& \int_{t_1}^{t_2}\Big( \eta(t)A_2(t)- \Psi_\Upsilon(v)(t,x_2) \int_{t_1}^{t}\eta(\tau)\mathrm{d}\tau\Big) \mathrm{d}t\nonumber\\
& \le M(\alpha,d,\beta)\Big(1+(t_2-t_1)^{-1-\frac{d}{\beta}}\Big),\label{HarnackEnd}
\end{align}
where the constant $M(\alpha,d,\beta)$ can be specified explicitly.

We now choose $s$ in \eqref{HarnackStart2} such that $f(s)=\min_{t\in [t_1,t_2]} f(t)$.
Using \eqref{fmin}, \eqref{A1Estimate} and \eqref{HarnackEnd} we obtain
\begin{align*}
\log \frac{u(t_1,x_1)}{u(t_2,x_2)} \le C_{LY}\log\big( \frac{t_2}{t_1}\big)+1+M(\alpha,d,\beta)\Big(1+(t_2-t_1)^{-1-\frac{d}{\beta}}\Big),
\end{align*}
and thus
\[
u(t_1,x_1)\le u(t_2,x_2) \Big(\frac{t_2}{t_1}\Big)^{C_{LY}}\exp\Big(1+M(\alpha,d,\beta)\Big[1+(t_2-t_1)^{-1-\frac{d}{\beta}}\Big]\Big),
\]
which implies the assertion in the case $|x_1-x_2|\le 1$. The general case follows from a classical scaling argument using that for any $\lambda>0$ the function $u(\lambda^\beta t, \lambda x)$ is a solution to the fractional heat equation on 
$(0,\infty)\times \iR^d$ if and only if $u(t,x)$ enjoys this property.
\end{proof}

{\bf Declarations:} On behalf of all authors, the corresponding author states that there is no conflict of interest. Data availability
statement: not applicable.



\begin{thebibliography}{99}
{\footnotesize
\bibitem{ACM}
 Alonso-Or\'an, D.; C\'ordoba, A.; Mart\'inez, \'Angel D.: Integral representation for fractional Laplace-Beltrami operators. Adv.\ Math.\ {\bf 328} (2018), 436--445.
\bibitem{AB}
Auchmuty, G.; Bao, D.: Harnack-type inequalities for evolution equations. 
Proc.\ Amer.\ Math.\ Soc.\ {\bf 122} (1994), 117--129.
\bibitem{BBG}
Bakry, D.; Bolley, F.; Gentil, I.: The Li-Yau inequality and applications under a curvature-dimension condition. Ann. Inst. Fourier {\bf 67} (2017), no. 1, 397--421.
\bibitem{BGL}
Bakry, D.; Gentil, I.; Ledoux, M.:
{\em Analysis and geometry of Markov diffusion operators}. 
Springer, 2014.
\bibitem{BGS}
Banica, V.; Gonz\'alez, M.; S\'aez, M.:
Some constructions for the fractional Laplacian on noncompact manifolds.
Rev. Mat. Iberoam. {\bf 31} (2015), 681--712.
\bibitem{BPSV}
Barrios, B.; Peral, I.; Soria, F.; Valdinoci, E.:
A Widder's type theorem for the heat equation with nonlocal diffusion. 
Arch. Ration. Mech. Anal. {\bf 213} (2014), no. 2, 629--650.
\bibitem{BL2}
Bass, R.\ F.; Levin, D.\ A.: Transition probabilities for symmetric jump processes. Trans.\ Amer.\ Math.\ Soc.\ {\bf 354}
(2002), 2933--2953.
\bibitem{BHLLMY15}
Bauer, F.; Horn, P.; Lin, Y.;  Lippner, G.; Mangoubi, D.; Yau, S.-T.:
 {Li-{Y}au inequality on graphs}, J.\ Differential Geom.\ {\bf 99} (2015), 359--405.
\bibitem{BlGe}
Blumenthal, R. M.; Getoor, R. K.:
Some theorems on stable processes.
Trans. Amer. Math. Soc. {\bf 95} (1960), 263--273.
\bibitem{BoSiVa}
Bonforte, M.; Sire, Y.; V\'azquez, J.\ L.: Optimal existence and uniqueness theory for the fractional heat equation. Nonlinear Anal. {\bf 153} (2017), 142--168.
\bibitem{CDFGV}
Cavaterra, C.; Dipierro, S.; Farina, A.; Gao, Z.; Valdinoci, E.: Pointwise gradient bounds for entire solutions of elliptic equations with non-standard growth conditions and general nonlinearities. J.\ Differential Equations {\bf 270} (2021), 435--475.
\bibitem{CLD}
Chang-Lara, H.\ A.; D'avila, G.: H\"older estimates for non-local parabolic equations with critical drift. J. Differential
Equations {\bf 260} (2016), 4237--4284.
\bibitem{CK}
Chen, Z.-Q.; Kumagai, T.: Heat kernel estimates for stable-like processes on d-sets. Stochastic Process. Appl.
{\bf 108} (2003), 27--62.
\bibitem{CRSTV}
Ciaurri, \'O.; Roncal, L.; Stinga, P.\ R.; Torrea, J.\ L.; Varona, J.\ L.\:
Nonlocal discrete diffusion equations and the fractional discrete Laplacian, regularity and applications.
Adv.\ Math.\ {\bf 330} (2018), 688--738.
\bibitem{DKSZ}
Dier, D.; Kemppainen, J.; Siljander, J.; Zacher, R.: On the parabolic Harnack inequality for non-local diffusion equations. 
Math.\ Z.\ {\bf 295} (2020), 1751--1769.
\bibitem{DKZ}
Dier, D.; Kassmann, M.; Zacher, R.: Discrete versions of the Li-Yau gradient estimate.
Ann.\ Sc.\ Norm.\ Super.\ Pisa Cl.\ Sci.\ (5),
{\bf 22} (2021), 691--744.
\bibitem{Gar}
Garofalo, N.: Fractional thoughts. New developments in the analysis of nonlocal operators, 1--135, Contemp. Math.,  {\bf 723}, Amer. Math. Soc., Providence, RI, 2019.
\bibitem{GhKa}
Ghosh, T.; Kassmann, M.: Pointwise inequalities for solutions to the fractional heat equation. In preparation.
\bibitem{Jac}
Jacob, N.:
Pseudo differential operators \& Markov processes. Vol. II.
Generators and their potential theory. Imperial College Press, London, 2002.
\bibitem{KimLim}
Kim, K.-H.; Lim, S.:
Asymptotic behaviors of fundamental solution and its derivatives to fractional diffusion-wave equations. 
J. Korean Math. Soc. {\bf 53} (2016), no. 4, 929--967.
\bibitem{LY}
Li, P.; Yau, S.-T.: On the parabolic kernel of the Schr\"odinger operator. Acta. Math. {\bf 156} (1986), 153--201.
\bibitem{Mun}
M\"unch, F.: Li-Yau inequality on finite graphs via non-linear curvature dimension conditions. J. Math. Pures Appl. {\bf 120} (2018), 130--164.
\bibitem{Nor}
Norris, J.: {\em Markov Chains}. Cambridge Series in Statistical and Probabilistic Mathematics. Cambridge University Press,  1997.
\bibitem{SWZ}
Spener, A.; Weber, F.; Zacher, R.: The fractional Laplacian has infinite dimension. Comm. Partial Differential Equations {\bf 45} (2020), no. 1, 57--75.
\bibitem{WZ}
Weber, F.; Zacher, R.: The entropy method under curvature-dimension conditions in the spirit of Bakry-\'Emery in the discrete setting of Markov chains. J.\ Funct.\ Anal.\ {\bf 281} (2021), Paper No. 109061, 81 pp.
\bibitem{VPQR}
V\'azquez, J.; de Pablo, A.; Quirós, F.; Rodriguez, A.:
Classical solutions and higher regularity for nonlinear fractional diffusion equations.  
J. Eur. Math. Soc. (JEMS) {\bf 19} (2017), no. 7, 1949--1975.
}
\end{thebibliography}
\end{document}